\documentclass[letterpaper,10pt,conference]{ieeeconf}

\usepackage{algorithm,algorithmic,amssymb,amsthm,cite,graphicx,multirow,threeparttable,url}
\usepackage[tight,footnotesize]{subfigure}
\usepackage[usenames]{color}
\usepackage[cmex10]{amsmath}
\IEEEoverridecommandlockouts

\newcommand{\cA}{\mathcal{A}}

\newcommand{\cB}{\mathcal{B}}

\newcommand{\E}{\mathbb{E}}

\newcommand{\cI}{\mathcal{I}}

\newcommand{\cK}{\mathcal{K}}
\newcommand{\whcK}{\widehat{\cK}}
\newcommand{\wtcK}{\widetilde{\cK}}

\newcommand{\N}{\mathbb{N}}

\newcommand{\R}{\mathbb{R}}

\newcommand{\wtX}{\widetilde{X}}

\newcommand{\rmy}{\mathrm{y}}

\newcommand{\veps}{\varepsilon}

\newcommand{\nN}[1]{\left[\!\left[{#1}\right]\!\right]}

\newcommand{\tT}{\mathrm{T}}

\newcommand{\FDP}{\textsf{\textsc{fdp}}}
\newcommand{\NDP}{\textsf{\textsc{ndp}}}
\newcommand{\FDR}{\textsf{\textsc{fdr}}}

\theoremstyle{plain}
\newtheorem{lemma}{Lemma}
\newtheorem{theorem}{Theorem}

\newtheorem{proposition}{Proposition}

\theoremstyle{definition}
\newtheorem{definition}{Definition}
\theoremstyle{remark}
\newtheorem{remark}{Remark}

\begin{document}
\title{\LARGE\bf%
    Group Model Selection Using Marginal Correlations:\\%
    The Good, the Bad and the Ugly}

\author{\parbox{3in}{\centering Waheed U. Bajwa\\
         Dept. of Electrical and Computer Engineering\\
            Rutgers, The State University of New Jersey\\
            {\tt waheed.bajwa@rutgers.edu}}%
         \hspace*{0.5in}
         \parbox{3in}{\centering Dustin G. Mixon\\
            Dept. of Mathematics and Statistics\\
            Air Force Institute of Technology\\
            {\tt dustin.mixon@afit.edu}}%
\thanks{This work and the first author are supported in part by the National Science Foundation under grant CCF-1218942.}%
}


\maketitle

\begin{abstract}
Group model selection is the problem of determining a small subset of groups
of predictors (e.g., the expression data of genes) that are responsible for
majority of the variation in a response variable (e.g., the malignancy of a
tumor). This paper focuses on group model selection in high-dimensional
linear models, in which the number of predictors far exceeds the number of
samples of the response variable. Existing works on high-dimensional group
model selection either require the number of samples of the response variable
to be significantly larger than the total number of predictors contributing
to the response or impose restrictive statistical priors on the predictors
and/or nonzero regression coefficients. This paper provides comprehensive
understanding of a low-complexity approach to group model selection that
avoids some of these limitations. The proposed approach, termed \emph{Group
Thresholding} (GroTh), is based on thresholding of marginal correlations of
groups of predictors with the response variable and is reminiscent of
existing thresholding-based approaches in the literature. The most important
contribution of the paper in this regard is relating the performance of GroTh
to a polynomial-time verifiable property of the predictors for the general
case of arbitrary (random or deterministic) predictors and arbitrary nonzero
regression coefficients.
\end{abstract}


\section{Introduction}
\subsection{Motivation and Background}
One of the most fundamental of problems in statistical data analysis is to
learn the relationship between the samples of a \emph{dependent} or
\emph{response} variable (e.g., the malignancy of a tumor, the health of a
network) and the samples of \emph{independent} or \emph{predictor} variables
(e.g., the expression data of genes, the traffic data in the network). This
problem was relatively easy in the data-starved world of yesteryears. We had
$n$ samples and $p$ predictors, and our inability to observe too many
variables meant that we lived in the ``$n$ greater than or equal to $p$''
world. Times have changed now. The data-rich world of today has enabled us to
simultaneously observe an unprecedented number of variables per sample. It is
nearly impossible in many of these instances to collect as many, or more,
samples as the number of predictors. Imagine, for example, collecting
hundreds of thousands of thyroid tumors in a clinical setting. The ``$n$
smaller than $p$'' world is no longer a theoretical construct in statistical
data analysis. It has finally arrived; and it is here to stay.

This paper concerns statistical inference in the ``$n$ smaller than $p$''
setting for the case when the response variable depends linearly on the
predictors. Mathematically, a model of this form can be expressed as
\begin{align}
    y_i = \sum_{j=1}^p
        x_{i,j}\beta^0_j + \veps_i, \ i=1,\dots,n.
\end{align}
Here, $y_i$ denotes the $i$-th sample of the response variable, $x_{i,j}$
denotes the $i$-th sample of the $j$-th predictor, $\veps_i$ denotes the
error in the model, and the parameters $\{\beta^0_j\}$ are called
\emph{regression coefficients}. This relationship between the samples of the
response variable and those of the predictors can be expressed compactly in
matrix-vector form as $\rmy = X \beta^0 + \veps$. The matrix $X$ in this
form, termed the \emph{design matrix}, is an $n \times p$ matrix whose $j$-th
column comprises the $n$ samples of the $j$-th predictor. In tumor
classification, for example, an entry in the response variable $\rmy$ could
correspond to the malignancy (expressed as a numerical number) of a tumor
sample, while the corresponding row in $X$ would correspond to the expression
level of $p$ genes in that tumor sample.

The linear model $\rmy = X\beta^0 + \veps$, despite its mathematical
simplicity, continues to make profound impacts in countless application areas
\cite{Rencher.Schaalje.Book2008}. Such models are used for various
inferential purposes. In this paper, we focus on the problem of \emph{model
selection} in high-dimensional linear models, which involves determining a
small subset of $p$ predictors that are responsible for majority (or all) of
the variation in the response variable $\rmy$. High-dimensional model
selection can be used to implicate a small number of genes in the development
of cancerous tumors, identify a small number of genes that primarily affect
prognosis of a disease, etc.

\begin{algorithm*}[t]
\caption{The Group Thresholding (GroTh) Algorithm for Group Model Selection}
\label{alg:groth}
\textbf{Input:} An $n \times p$ design matrix $X$, response variable $y$, number of predictors per group $r$, and (group) model order $k$\\
\textbf{Output:} An estimate $\whcK \subset \{1,\dots,m\}$ of the true
(group) model $\cK$
\begin{algorithmic}
\STATE $f \leftarrow \begin{bmatrix}X_1 & X_2 & \dots & X_m\end{bmatrix}^\tT
y$ \hfill \COMMENT{Compute marginal correlations}

\STATE $\big(\cI, \big\{\|f_{(j)}\|_2\big\}\big) \leftarrow
\text{SORT}\Big(\Big(\big\{1,\dots,m\big\},\big\{\|f_i\|_2 := \|X_i^\tT
y\|_2\big\}\Big)\Big)$ \hfill \COMMENT{Sort groups of marginal correlations}

\STATE $\whcK \leftarrow \cI[1:k]$ \hfill \COMMENT{Select model via group
thresholding}
\end{algorithmic}
\end{algorithm*}

\subsection{Group Model Selection and Our Contributions}
There exist many applications in statistical model selection where the
implication of a single predictor in the response variable implies presence
of other related predictors in the true model. This happens, for instance, in
the case of microarray data when the genes (predictors) share the same
biological pathway \cite{Segal.etal.JCB2004}. In such situations, it is
better to reformulate the problem of model selection in a ``group'' setting.
Specifically, the response variable $\rmy = X\beta^0 + \veps$ in
high-dimensional linear models in group settings can be best explained by a
small number of \emph{groups} of predictors:
\begin{align}
\rmy = \sum_{i=1}^{m} X_i \beta_i^0 + \veps = \sum_{i \in \cK} X_i \beta_i^0
+ \veps,
\end{align}
where $X_i$, an $n \times p_i$ submatrix of $X$, denotes the $i$-th group of
predictors, $\beta_i^0$ denotes the group of $p_i$ regression coefficients
associated with the group of predictors $X_i$, and the set $\cK := \{1 \leq i
\leq m : \beta^0_i \not= 0\}$ denotes the underlying true (group) model,
corresponding to the $k := |\cK| \ll m$ groups of predictors that explain
$\rmy$.

One of the main contributions of this paper is comprehensive understanding of
a polynomial-time algorithm, which we term \emph{Group Thresholding} (GroTh),
that returns an estimate $\whcK$ of the true (group) model $\cK$ for the
general case of arbitrary (random or deterministic) design matrices and
arbitrary nonzero regression coefficients. To this end, we make use of two
computable geometric measures of \emph{group coherence} of a design
matrix---the \emph{worst-case group coherence} $\mu_X^g$ and the
\emph{average group coherence} $\nu_X^g$---to provide a \emph{nonasymptotic}
analysis of GroTh (Algorithm~\ref{alg:groth}). We in particular establish
that if $X$ satisfies a verifiable \emph{group coherence property} then, for
all but a vanishingly small fraction of possible models $\cK$, GroTh: ($i$)
handles linear scaling of the total number of predictors contributing to the
response, $\sum_{i \in \cK} p_i = O(n)$,\footnote{Recall that $f(n)=O(g(n))$
if there exists positive $C$ and $n_0$ such that for all $n>n_0$, $f(n)\leq
Cg(n)$. Also, $f(n)=\Omega(g(n))$ if $g(n)=O(f(n))$, and $f(n)=\Theta(g(n))$
if $f(n)=O(g(n))$ and $g(n)=O(f(n))$.} and ($ii$) returns indices of the
groups of predictors whose contributions to the response,
\{$\|\beta_i^0\|_2\}_{i \in \cK}$, are above a certain \emph{self-noise
floor} that is a function of both $\mu_X^g$ and $\|\beta^0\|_2$.

\subsection{Relationship to Previous Work}
The basic idea of using grouped predictors for inference in linear models has
been explored by various researchers in recent years. Some notable works in
this direction in the \mbox{``$n \ll p$''} setting include
\cite{Yuan.Lin.JRSSSB2006,Bach.JMLR2008,Nardi.Rinaldo.EJS2008,Huang.Zhang.AS2010,Eldar.etal.ITSP2010,Ben-Haim.Eldar.IJSTSP2011,Elhamifar.Vidal.ITSP2012}.
Despite these inspiring results, more work needs to be done for
high-dimensional group model selection. This is because the results reported
in
\cite{Yuan.Lin.JRSSSB2006,Bach.JMLR2008,Nardi.Rinaldo.EJS2008,Huang.Zhang.AS2010,Eldar.etal.ITSP2010,Ben-Haim.Eldar.IJSTSP2011,Elhamifar.Vidal.ITSP2012}
do not guarantee linear scaling of the total number of predictors
contributing to the response for the case of arbitrary design matrices and
nonzero regression coefficients.

The work in this paper is also related to another body of work in statistics
and signal processing literature that studies the high-dimensional linear
model $\rmy = X \beta^0 + \veps$ for the restrictive case of $X$ having a
Kronecker structure: $X := A^\tT \otimes I$ for some matrix $A$, where
$\otimes$ denotes the Kronecker product. An incomplete list of works in this
direction includes
\cite{Cotter.etal.ITSP2005,Tropp.etal.SP2006,Tropp.SP2006,Gribonval.etal.JFAA2008,Eldar.Rauhut.ITIT2010,Obozinski.etal.AS2011,Davies.Eldar.ITIT2012}.
These restrictive works, however, also fail to guarantee linear scaling of
the total number of predictors contributing to the response for the case of
arbitrary nonzero regression coefficients.

Finally, note that the group model selection procedure studied in this paper
is based on analyzing the \emph{marginal correlations}, $X^\tT \rmy$, of
predictors with the response variable. Therefore, our work is algorithmically
similar to the group thresholding approaches of
\cite{Gribonval.etal.JFAA2008,Eldar.Rauhut.ITIT2010,Ben-Haim.Eldar.IJSTSP2011}.
The main appeal of such approaches is their low computational complexity of
$O(np)$, which is much smaller than the typical computational complexity
associated with other model selection procedures \cite{Fan.Lv.JRSSSB2008}. In
addition to the scaling limitations of the total number of influential
predictors discussed earlier, however, the works in
\cite{Gribonval.etal.JFAA2008,Eldar.Rauhut.ITIT2010,Ben-Haim.Eldar.IJSTSP2011}
also incorrectly conclude that performance of thresholding-based approaches
is inversely proportional to the dynamic range, $\frac{\max_{i \in \cK}
\|\beta_i^0\|_2}{\min_{i \in \cK} \|\beta_i^0\|_2}$, of the nonzero groups of
regression coefficients.

\subsection{Mathematical Convention}
The predictors and the response variable are assumed to be real valued
throughout the paper, with the understanding that extensions to a
complex-valued setting can be carried out in a straightforward manner.
Uppercase letters are reserved for matrices, while lowercase letters are used
for both vectors and scalars. Constants that do not depend upon the problem
parameters (such as $n$, $m$, $p$, and $k$) are denoted by $c_0$, $c_1$, etc.
The notation $\nN{q}$ for $q \in \N$ is a shorthand for the set
$\{1,\dots,q\}$, while the notation $\overset{D}{=}$ signifies \emph{equality
in distribution}. The transpose operation is denoted by $(\cdot)^\tT$ and the
spectral norm of a matrix is denoted by $\|\cdot\|_2$. Finally, the
$\ell_{p,q}$ norm of a vector $v^\tT = \begin{bmatrix}v_1^\tT & \dots &
v_m^\tT\end{bmatrix}$ with each $v_i \in \R^r$ is defined as $\|v\|_{p,q} :=
\left(\sum_{i=1}^{m} \|v_i\|^q_p\right)^{1/q}$ for $p, q \in (0,\infty]$,
where $\|\cdot\|_p$ denotes the usual $\ell_p$ norm. Note that
$\|v\|_{p,\infty} \equiv \max_i \|v_i\|_p$ and $\|v\|_{p,q} \equiv \|v\|_q$
for $r=1$.

\subsection{Organization}
In Section~\ref{sec:mainres}, we mathematically formulate the problem of
group model selection, rigorously define the notions of worst-case group
coherence, average group coherence and the group coherence property, and
state and discuss the main result of the paper. In Section~\ref{sec:proof},
we prove the main result of the paper. Finally, we present some numerical
results in Section~\ref{sec:num_res} and conclude in Section~\ref{sec:conc}.

\section{Group Model Selection Using GroTh}
\label{sec:mainres}

\subsection{Problem Formulation}
The object of attention in this paper is the high-dimensional linear model $y
= X \beta^0 + \veps$ relating the response variable $y \in \R^n$ to the $p$
$( \gg n)$ predictors comprising the columns of the design matrix $X$. Since
scalings of the columns of $X$ can be absorbed into the regression vector
$\beta^0$, we assume without loss of generality that the columns of $X$ have
unit $\ell_2$ norms. There are three simplifying assumptions we make in this
paper that will be relaxed in a sequel to this work. First, the modeling
error is zero, $\veps = 0$, and thus the response variable is exactly equal
to a parsimonious linear combination of grouped predictors: $y = \sum_{i \in
\cK} X_i \beta^0_i$. Second, the groups of predictors $\{X_i\}_{i=1}^m$ are
characterized by the same number of predictors per group: $X_i \in \R^{n
\times r}$ with $r := \tfrac{p}{m} \leq n$. Third, the groups of predictors
$\{X_i\}_{i=1}^m$ are orthonormalized: $X_i^\tT X_i = I$.

The main goal of this paper is characterization of the performance of a group
model selection procedure, termed GroTh, that returns an estimate $\whcK$ of
the true model $\cK$ by sorting the groups of marginal correlations $f_i :=
X_i^\tT \rmy$ according to their $\ell_2$-norms, $\|f_i\|_2$, in descending
order and setting $\whcK$ to be indices of the first $k$ sorted groups of
marginal correlations (see Algorithm~\ref{alg:groth}). Instead of focusing on
the worst-case performance of GroTh, however, we seek to characterize its
performance for an arbitrary (but fixed) set of nonzero (grouped) regression
coefficients supported on \emph{most} models. Specifically, we do not impose
any statistical prior on the set of nonzero regression coefficients, while we
assume that the true (group) model $\cK := \{i \in \nN{m}: \beta_i^0  \not=
0\}$ is a uniformly random $k$-subset of $\nN{m}$. Finally, the metrics of
goodness we use in this paper are the \emph{false-discovery proportion}
(\FDP) and the \emph{non-discovery proportion} (\NDP), defined as
\begin{align}
    \FDP(\whcK) := \frac{|\whcK \setminus \cK|}{|\whcK|} \quad \text{and} \quad \NDP(\whcK) := \frac{|\cK \setminus \whcK|}{|\cK|},
\end{align}
respectively. These two metrics have gained widespread usage in multiple
hypotheses testing problems in recent years. In particular, the expectation
of the \FDP~is the well-known \emph{false-discovery rate} (\FDR)
\cite{Benjamini.Hochberg.JRSSSB1995,Abramovich.etal.AS2006}.

\subsection{Main Result and Discussion}
Heuristically, successful group model selection requires the groups of
predictors contributing to the response variable to be sufficiently
distinguishable from the ones outside the true model. In this paper, we
capture the notion of distinguishability of predictors through two easily
computable, global geometric measures of the design matrix, namely, the
worst-case group coherence and the average group coherence. The worst-case
group coherence of $X$ is defined as
\begin{align}
    \mu_X^g := \max_{i,j \in \nN{m}:i\not=j} \|X_i^\tT X_j\|_2,
\end{align}
while the average group coherence of $X$ is defined as
\begin{align}
    \nu_X^g := \frac{1}{m-1} \max_{i \in \nN{m}} \Big\|\sum_{j \in \nN{m}: j\not=i} X_i^\tT X_j\Big\|_2.
\end{align}
Note that $\mu_X^g$ is a trivial upper bound on $\nu_X^g$. It is also worth
pointing that the worst-case group coherence and its variants have existed in
earlier literature \cite{Eldar.etal.ITSP2010,Bajwa.etal.TR2010}, but the
average group coherence is defined for the first time in here.

The central thesis of this paper is that group model selection using GroTh
can be successful if these two measures of group coherence of $X$ are small
enough. In particular, we address the question of how small should these two
measures be in terms of the \emph{group coherence property}.
\begin{definition}[The Group Coherence Property]
The $n \times rm$ design matrix $X$ is said to satisfy the group coherence
property if the following two conditions hold for some positive constants
$c_\mu$ and $c_\nu$:
\begin{align}
\tag{GroCP-1}\label{eqn:GroCP-1}
    \mu_X^g &\leq \frac{c_\mu}{\sqrt{\log{m}}} \quad \text{and}\\
\tag{GroCP-2}\label{eqn:GroCP-2}
    \nu_X^g &\leq c_\nu \mu_X^g \sqrt{\frac{r \log{m}}{n}}.
\end{align}
\end{definition}
It is straightforward to observe from the above definition that the group
coherence property is a global property of $X$ that can be explicitly
verified in polynomial time. Finally, we define $\beta^0_{(\ell)}$ to be the
$\ell$-th largest group of nonzero regression coefficients:
$\|\beta^0_{(1)}\|_2 \geq \|\beta^0_{(2)}\|_2 \geq \dots \geq
\|\beta^0_{(k)}\|_2 > 0$. We are now ready to state the main result of this
paper.

\begin{theorem}[Group Model Selection Using GroTh]\label{thm:groth}
Suppose the design matrix $X$ satisfies the group coherence property with
parameters $c_\mu$ and $c_\nu$. Next, fix parameters $c_1 \geq 2$, $c_2 \in
(0,1)$, and define parameter $c_3 := \frac{32\sqrt{2e}(2 c_1 - 1)}{(1 -
c_2)(c_1 - 1)}$. Then, under the assumptions $c_1 rk \leq n$, $c_\mu <
c_3^{-1}$, and $c_\nu \leq \sqrt{c_1} c_2 c_3$, we have with probability
exceeding $1 - e^2m^{-1}$ that
\begin{align}
\label{eqn:thm_grothset}
\Big\{i \in \cK : \|\beta_i^0\|_2 \geq c_3 \mu_X^g \|\beta^0\|_2 \sqrt{\log{m}}\Big\} \subset \whcK,
\end{align}
resulting in $\FDP(\whcK) \leq 1 - L/k$ and $\NDP(\whcK) \leq 1 - L/k$, where
$L$ is defined to be the largest integer for which the inequality
$\|\beta_{(L)}^0\|_2 \geq c_3 \mu_X^g \|\beta^0\|_2 \sqrt{\log{m}}$ holds.
Here, the probability is with respect to the uniform distribution of the true
model $\cK$ over all possible models.
\end{theorem}

%

A proof of this theorem is given in Section~\ref{sec:proof}. We now provide a
brief discussion of the significance of this result. First,
Theorem~\ref{thm:groth} indicates that a polynomial-time verifiable property,
namely, the group coherence property, of the design matrix can be checked to
ascertain whether GroTh, which has computational complexity of $O(np)$, is
well suited for group model selection. Second, it states that if $X$
satisfies the group coherence property then GroTh handles linear scaling of
the total number of predictors contributing to the response, $rk = O(n)$, for
all but a vanishingly small fraction $O(m^{-1})$ of models. This is in stark
contrast to the earlier works
\cite{Gribonval.etal.JFAA2008,Eldar.Rauhut.ITIT2010,Ben-Haim.Eldar.IJSTSP2011}
on thresholding-based approaches in high-dimensional linear models, which do
not guarantee such linear scaling for the case of arbitrary nonzero
regression coefficients. Note that while we do not provide in this paper
explicit examples of design matrices satisfying the group coherence property,
numerical results in Section~\ref{sec:num_res} show that the set of design
matrices satisfying the group coherence property is not empty.

Finally, Theorem~\ref{thm:groth} offers a nice interpretation of the price
one might have to pay in estimating the true model using only marginal
correlations. Specifically, \eqref{eqn:thm_grothset} in the theorem implies
group thresholding of marginal correlations effectively gives rise to a
\emph{self-noise floor} of $O\left(\mu_X^g \|\beta^0\|_2
\sqrt{\log{m}}\right)$. In words, the estimate $\whcK$ returned by GroTh is
guaranteed to return the indices of all the groups of predictors whose
contributions to the response variable (in the $\ell_2$ sense) are above the
self-noise floor of $O\left(\mu_X^g \|\beta^0\|_2 \sqrt{\log{m}}\right)$
(cf.~\ref{eqn:thm_grothset}). This is again a significant improvement over
the earlier works
\cite{Gribonval.etal.JFAA2008,Eldar.Rauhut.ITIT2010,Ben-Haim.Eldar.IJSTSP2011},
which suggest that performance of thresholding-based approaches is inversely
proportional to the dynamic range, $\frac{\max_{i \in \cK}
\|\beta_i^0\|_2}{\min_{i \in \cK} \|\beta_i^0\|_2}$, of the nonzero groups of
regression coefficients. In order to expand on this, we observe from
\eqref{eqn:thm_grothset} that
\begin{align}
\nonumber
    \|\beta_i^0\|_2 &= \Omega\left(\mu_X^g \|\beta^0\|_2 \sqrt{\log{m}}\right)\\
\label{eqn:LAR_eqn}
    \Longleftrightarrow \quad \frac{\|\beta_i^0\|^2_2}{\|\beta^0\|^2_2/k} &=
\Omega\left(k\,(\mu_X^g)^2 \log{m}\right).
\end{align}
Theorem~\ref{thm:groth} and the left-hand side of \eqref{eqn:LAR_eqn}
indicate that inclusion of the $i$-th group of predictors in the estimate
$\whcK$ is in fact related to the ratio of the \emph{energy contributed by
the $i$-th group of predictors} to the \emph{average energy contributed per
group of nonzero predictors}: $\frac{\|\beta_i^0\|^2_2}{\|\beta^0\|^2_2/k}$.
Further, this implies that an increase in the dynamic range that comes from a
decrease in $\min_{i \in \cK} \|\beta_i^0\|_2$ cannot affect the performance
of GroTh too much since $\frac{\|\beta_i^0\|^2_2}{\|\beta^0\|^2_2/k}$
increases for most groups of predictors in this case. This is indeed
confirmed by the numerical experiments reported in Section~\ref{sec:num_res}.

\section{Proof of the Main Result}
\label{sec:proof}

We begin by developing some notation to facilitate the forthcoming analysis.
Notice that the $p$-dimensional vector of marginal correlations, $f = X^\tT
y$, can be written as $m$ groups of $r$-dimensional marginal correlations:
$f^\tT = \begin{bmatrix}f_1^\tT & \dots & f_m^\tT\end{bmatrix}$ with the $r
\times 1$ vector $f_i = X_i^\tT y$. In the following, we use $X_\cK$ (an $n
\times rk$ submatrix of $X$), $\beta_\cK^0$ (an $rk \times 1$ subvector of
$\beta^0$), and $f_\cK := X_\cK^\tT y = X_\cK^\tT X_\cK \beta_\cK^0$ (an $rk
\times 1$ subvector of $f$) to denote the groups of predictors, groups of
regression coefficients, and the marginal correlations corresponding to the
true model $\cK$, respectively. Similarly, we use $X_{\cK^c}$ and $f_{\cK^c}
:= X_{\cK^c}^\tT y = X_{\cK^c}^\tT X_{\cK} \beta_\cK^0$ to denote the groups
of predictors and the marginal correlations corresponding to the complement
set $\cK^c := \nN{m} \setminus \cK$, respectively.

\subsection{Lemmata}
Proof of Theorem~\ref{thm:groth} requires understanding the behaviors of the
$rk \times 1$ group vector $\big(X_\cK^\tT X_\cK - I\big)\beta_\cK^0$ and the
$r(m-k) \times 1$ group vector $X_{\cK^c}^\tT X_{\cK} \beta_\cK^0$. In this
subsection, we state and prove two lemmas that help us toward this goal. We
will then leverage these two lemmas to provide a proof of
Theorem~\ref{thm:groth}.

Before proceeding, recall that $\cK$ is taken to be a uniformly random
$k$-subset of $\nN{m}$, while the set of nonzero group regression
coefficients $\{z_i\}_{i=1}^{k} := \{\beta^0_i : i \in \cK\}$ is considered
to be deterministic (and fixed) but unknown. It therefore follows that the
$rk$-dimensional group vector $\big(X_\cK^\tT X_\cK - I\big)\beta_\cK^0$ can
be equivalently expressed as
\begin{align}
    \big(X_\cK^\tT X_\cK - I\big)\beta_\cK^0 \overset{D}{=} \big(X_\Pi^\tT X_\Pi - I\big) z,
\end{align}
where $\bar{\Pi} := (\pi_1,\dots,\pi_m)$ is a random permutation of $\nN{m}$,
$\Pi := (\pi_1,\dots,\pi_k)$ denotes the first $k$ elements of $\bar{\Pi}$,
$X_\Pi := \begin{bmatrix}X_{\pi_1} & \dots & X_{\pi_k}\end{bmatrix}$ is an $n
\times rk$ submatrix of $X$, and $z^\tT := \begin{bmatrix}z_1^\tT & \dots &
z_k^\tT\end{bmatrix}$ is an $rk \times 1$ (group) vector of nonzero
regression coefficients. Similarly, the $r(m-k)$-dimensional group vector
$X_{\cK^c}^\tT X_{\cK} \beta_\cK^0$ can be expressed as
\begin{align}
    X_{\cK^c}^\tT X_{\cK} \beta_\cK^0 \overset{D}{=} X_{\Pi^c}^\tT X_{\Pi} z
\end{align}
where $\Pi^c := (\pi_{k+1},\dots,\pi_m)$ denotes the last $m-k$ elements of
$\bar{\Pi}$ and $X_{\Pi^c} := \begin{bmatrix}X_{\pi_{k+1}} & \dots &
X_{\pi_m}\end{bmatrix}$ is an $n \times r(m-k)$ submatrix of $X$.

\begin{lemma}\label{lemma:GroStOC1}
Fix $c_1 \geq 2$ and $\epsilon \in (0,1)$. Next, assume $k \leq
\min\{\epsilon^2 (\nu_X^g)^{-2} + 1, c_1^{-1}m\}$ and let $\Pi =
(\pi_1,\dots,\pi_k)$ denote the first $k$ elements of a random permutation of
$\nN{m}$. Then for any fixed $rk \times 1$ group vector $z^\tT :=
\begin{bmatrix}z_1^\tT & \dots & z_k^\tT\end{bmatrix}$
\begin{align}
\nonumber
    &\Pr\Big(\big\|\big(X_{\Pi}^\tT X_{\Pi} - I\big)z\big\|_{2,\infty} \geq \epsilon \|z\|_2\Big)\\
    &\qquad\quad\leq e^2 k \exp\Big(-c_4 \big(\epsilon - \nu_X^g \sqrt{k-1}\big)^2 (\mu_X^g)^{-2}\Big),
\end{align}
where $c_4 := \frac{(c_1-1)^2}{1024e(2c_1 - 1)^2}$ is an absolute constant.
\end{lemma}
\begin{proof}
The proof of this lemma relies heavily on Banach-space-valued Azuma's
inequality stated in the Appendix.  To begin, note that
\begin{align}
    \|\left(X_{\Pi}^\tT X_{\Pi} - I\right)z\|_{2,\infty} \equiv \max_{i \in \nN{k}} \big\|\sum_{\substack{j=1\\j\not=i}}^{k} X_{\pi_i}^\tT X_{\pi_j} z_j\big\|_2.
\end{align}
We next fix an $i \in \nN{k}$ and define the event $\cA_i^\prime := \{\pi_i =
i^\prime\}$ for $i^\prime \in \nN{k}$. Then conditioned on $\cA_i^\prime$, we
have
\begin{align}
    \nonumber
    &\Pr\Big(\big\|\sum_{\substack{j=1\\j\not=i}}^{k} X_{\pi_i}^\tT X_{\pi_j} z_j\big\|_2 \geq \epsilon \|z\|_2 \Big| \cA_i^\prime\Big)\\
\label{eqn:pflemma1_conc}
    &\qquad\qquad\quad= \Pr\Big(\big\|\sum_{\substack{j=1\\j\not=i}}^{k} X_{i^\prime}^\tT X_{\pi_j} z_j\big\|_2 \geq \epsilon \|z\|_2 \Big| \cA_i^\prime\Big).
\end{align}

In order to make use of the concentration inequality in
Proposition~\ref{prop:azumaineq} in the Appendix for upper bounding
\eqref{eqn:pflemma1_conc}, we construct an $\R^r$-valued Doob martingale on
$\sum_{j\not=i} X_{i^\prime}^\tT X_{\pi_j} z_j$. We first define $\Pi^{-i} :=
(\pi_1,\dots,\pi_{i-1},\pi_{i+1},\dots,\pi_k)$ and then define the Doob
martingale $(M_0,M_1,\dots,M_{k-1})$ as follows:
\begin{align*}
    M_0 &:= \sum_{\substack{j=1\\j\not=i}}^{k} X_{i^\prime}^\tT \E\big[X_{\pi_j}\big|\cA_i^\prime\big]z_j, \quad \text{and}\\
    M_\ell &= \sum_{\substack{j=1\\j\not=i}}^{k} X_{i^\prime}^\tT \E\big[X_{\pi_j}\big|\pi_{1\rightarrow\ell}^{-i},\cA_i^\prime\big]z_j,\ \ell=1,\dots,k-1,
\end{align*}
where $\pi_{1\rightarrow\ell}^{-i}$ denotes the first $\ell$ elements of
$\Pi^{-i}$. The next step involves showing that the constructed martingale
has bounded $\ell_2$ differences. In order for this, we use $\pi_\ell^{-i}$
to denote the $\ell$-th element of $\Pi^{-i}$ and define
\begin{align}
    M_\ell(u) := \sum_{\substack{j=1\\j\not=i}}^{k} X_{i^\prime}^\tT \E\big[X_{\pi_j}\big|\pi_{1\rightarrow\ell-1}^{-i},\pi_{\ell}^{-i}=u,\cA_i^\prime\big]z_j
\end{align}
for $u \in \nN{m}$ and $\ell=1,\dots,k-1$. It can then be established using
techniques very similar to the ones used in the \emph{method of bounded
differences} for scalar-valued martingales that
\cite{McDiarmid.SiC1989,Motwani.Raghavan.Book1995}
\begin{align}
\label{eqn:pflemma1_sup_MOBD}
    \|M_\ell - M_{\ell-1}\|_2 \leq \sup_{u,v} \|M_\ell(u) - M_\ell(v)\|_2.
\end{align}

In order to upper bound $\|M_\ell(u) - M_\ell(v)\|_2$, we first define an $n
\times r$ random matrix
\begin{align}
\nonumber
    \wtX_{\ell,j}^{u,v} := &\E\big[X_{\pi_j}\big|\pi_{1\rightarrow\ell-1}^{-i},\pi_{\ell}^{-i}=u,\cA_i^\prime\big]\\
        &\qquad\qquad - \E\big[X_{\pi_j}\big|\pi_{1\rightarrow\ell-1}^{-i},\pi_{\ell}^{-i}=v,\cA_i^\prime\big].
\end{align}
Next, we notice that for every $j > \ell + 1, j \not= i$, the random variable
$\pi_j$ conditioned on
$\{\pi_{1\rightarrow\ell-1}^{-i},\pi_{\ell}^{-i}=u,\cA_i^\prime\}$ has a
uniform distribution over
$\nN{m}\setminus\{\pi_{1\rightarrow\ell-1}^{-i},u,i^\prime\}$, while $\pi_j$
conditioned on
$\{\pi_{1\rightarrow\ell-1}^{-i},\pi_{\ell}^{-i}=v,\cA_i^\prime\}$ has a
uniform distribution over
$\nN{m}\setminus\{\pi_{1\rightarrow\ell-1}^{-i},v,i^\prime\}$. Therefore, we
get
\begin{align}
    \wtX_{\ell,j}^{u,v} = \frac{1}{m-\ell-1} \left(X_u - X_v\right), \quad j > \ell+1, j\not=i.
\end{align}
In order to evaluate $\wtX_{\ell,j}^{u,v}$ for $j \leq \ell+1, j \not= i$, we
consider three cases for the index $i$. In the first case of $i \leq \ell$,
it can be seen that $\wtX_{\ell,j}^{u,v} = 0$ for every $j \leq \ell$ and
$\wtX_{\ell,j}^{u,v} = X_u - X_v$ for $j = \ell + 1$. In the second case of
$i = \ell + 1$, it can similarly be seen that $\wtX_{\ell,j}^{u,v} = 0$ for
every $j < \ell$ and $j = \ell + 1$, while $\wtX_{\ell,j}^{u,v} = X_u - X_v$
for $j = \ell$. In the final case of $i > \ell + 1$, it can be argued that
$\wtX_{\ell,j}^{u,v} = 0$ for every $j < \ell$, $\wtX_{\ell,j}^{u,v} = X_u -
X_v$ for $j = \ell$, and $\wtX_{\ell,j}^{u,v} = \frac{1}{m-\ell-1} \left(X_u
- X_v\right)$ for $j = \ell + 1$. Consequently, regardless of the initial
choice of $i$, we have
\begin{align}
\nonumber
    &\|M_\ell(u) - M_\ell(v)\|_2 \\
\nonumber
        &\quad\equiv \big\|\sum_{j\not=i} X_{i^\prime}^\tT \wtX_{\ell,j}^{u,v} z_j\big\|_2 \overset{(a)}{\leq} \sum_{\substack{j \geq \ell\\j\not=i}} \|X_{i^\prime}^\tT \wtX_{\ell,j}^{u,v}\|_2 \|z_j\|_2\\
\label{eqn:pflemma1_MOBD_diffs}
        &\quad\overset{(b)}{\leq} 2\mu_X^g \Big(\|z_\ell\|_2 + \|z_{\ell+1}\|_2 + \sum_{\substack{j > \ell + 1\\j\not=i}}\frac{\|z_j\|_2}{m-\ell-1}\Big),
\end{align}
where $(a)$ is due to the triangle inequality and the submultiplicative
nature of the induced norm, while $(b)$ primarily follows since
$\|X_{i^\prime}^\tT X_u - X_{i^\prime}^\tT X_v\|_2 \leq 2\mu_X^g$. We now
have from \eqref{eqn:pflemma1_sup_MOBD} and \eqref{eqn:pflemma1_MOBD_diffs}
that $\|M_\ell - M_{\ell-1}\|_2 \leq a_\ell$ with
\begin{align}
\label{eqn:pflemma1_diff_bds}
    a_\ell := 2\mu_X^g \Big(\|z_\ell\|_2 + \|z_{\ell+1}\|_2 + \sum_{\substack{j > \ell + 1\\j\not=i}}\frac{\|z_j\|_2}{m-\ell-1}\Big).
\end{align}

The next step needed to upper bound \eqref{eqn:pflemma1_conc} involves
providing an upper bound on $\|M_0\|_2$. To this end, note that
\begin{align}
\nonumber
    \|M_0\|_2 &\overset{(c)}{=} \big\|\sum_{j\not=i}X_{i^\prime}^\tT \Big(\frac{1}{m-1}\sum_{\substack{q=1\\q\not=i^\prime}}^{m} X_q\Big)z_j\big\|_2\\
\nonumber
        &\leq \Big\|\frac{1}{m-1}\sum_{\substack{q=1\\q\not=i^\prime}}^{m} X_{i^\prime}^\tT X_q\Big\|_2 \Big\|\sum_{j\not=i}z_j\Big\|_2\\
\label{eqn:pflemma1_M0_bd}
        &\overset{(d)}{\leq} \nu_X^g \sum_{j\not=i} \|z_j\|_2 \leq \nu_X^g \sqrt{k-1} \|z\|_2,
\end{align}
where $(c)$ follows since $\pi_j$ conditioned on $\cA_{i^\prime}$ has a
uniform distribution over $\nN{m} \setminus \{i^\prime\}$ and $(d)$ is a
consequence of the definition of average group coherence. Finally, we note
from \cite[Lemma~B.1]{Donahue.etal.CA1997} that $\rho_\cB(\tau)$ defined in
Proposition~\ref{prop:azumaineq} satisfies $\rho_\cB(\tau) \leq \tau^2/2$ for
$(\cB,\|\cdot\|) \equiv (L_2(\R^r), \|\cdot\|_2)$. Consequently, under the
assumption that $k \leq \epsilon^2 (\nu_X^g)^{-2} + 1$, it can be seen from
our construction of the Doob martingale that
\begin{align}
\nonumber
    &\Pr\Big(\big\|\sum_{\substack{j=1\\j\not=i}}^{k} X_{i^\prime}^\tT X_{\pi_j} z_j\big\|_2 \geq \epsilon \|z\|_2 \Big| \cA_i^\prime\Big)\\
\nonumber
    &\qquad\leq \Pr\Big(\big\|M_{k-1} - M_0\big\|_2 \geq \big(\epsilon - \nu_X^g \sqrt{k-1}\big) \|z\|_2 \Big| \cA_i^\prime\Big)\\
\label{eqn:pflemma1_conc_2}
    &\qquad\overset{(e)}{\leq} e^2 \exp\Bigg(-\frac{c_0 \big(\epsilon - \nu_X^g \sqrt{k-1}\big)^2 \|z\|^2_2}{\sum\limits_{\ell=1}^{k-1} a_\ell^2}\Bigg),
\end{align}
where $(e)$ follows from Banach-space-valued Azuma's inequality stated in the
Appendix. Further, it can be established using \eqref{eqn:pflemma1_diff_bds}
through tedious algebraic manipulations that
\begin{align}
\nonumber
    \sum_{\ell=1}^{k-1} a_\ell^2 &\leq \left(16 + \frac{4k^2}{(m-k)^2} + \frac{16k}{m-k}\right) (\mu_X^g)^2 \|z\|^2_2\\
\label{eqn:pflemma1_sum_diff_bds}
        &\overset{(f)}{\leq} 4(2+(c_1-1)^{-1})^2 (\mu_X^g)^2 \|z\|^2_2,
\end{align}
where $(f)$ follows from the condition $k \leq m/c_1$. Combining all these
facts together, we finally obtain from \eqref{eqn:pflemma1_conc_2} and
\eqref{eqn:pflemma1_sum_diff_bds} the following concentration inequality:
\begin{align}
\nonumber
    &\Pr\Big(\big\|\big(X_{\Pi}^\tT X_{\Pi} - I\big)z\big\|_{2,\infty} \geq \epsilon \|z\|_2\Big)\\
\nonumber
    &\qquad\overset{(g)}{\leq} k \Pr\Big(\big\|\sum_{\substack{j=1\\j\not=i}}^{k} X_{\pi_i}^\tT X_{\pi_j} z_j\big\|_2 \geq \epsilon \|z\|_2\Big)\\
\nonumber
    &\qquad= k \sum_{i^\prime=1}^{m} \Pr\Big(\big\|\sum_{\substack{j=1\\j\not=i}}^{k} X_{i^\prime}^\tT X_{\pi_j} z_j\big\|_2 \geq \epsilon \|z\|_2 \Big| \cA_i^\prime\Big) \Pr(\cA_i^\prime)\\
    &\qquad\overset{(h)}{\leq} e^2 k \exp\Big(-c_2 \big(\epsilon - \nu_X^g \sqrt{k-1}\big)^2 (\mu_X^g)^{-2}\Big),
\end{align}
where $c_4 := c_0/4(2+(c_1-1)^{-1})^2$, $(g)$ follows from the union bound
and the fact that $\pi_i$'s are identically distributed, while $(h)$ follows
since $\pi_i$ has a uniform distribution over the set $\nN{m}$.
\end{proof}

\begin{lemma}\label{lemma:GroStOC2}
Fix $c_1 \geq 2$ and $\epsilon \in (0,1)$. Next, assume $k \leq
\min\{\epsilon^2 (\nu_X^g)^{-2}, c_1^{-1}m\}$, and let $\Pi =
(\pi_1,\dots,\pi_k)$ and $\Pi^c = (\pi_{k+1},\dots,\pi_m)$ denote the first
$k$ elements and the last $(m-k)$ elements of a random permutation of
$\nN{m}$, respectively. Then for any fixed $rk \times 1$ group vector
\mbox{$z^\tT := \begin{bmatrix}z_1^\tT & \dots & z_k^\tT\end{bmatrix}$}
\begin{align}
\nonumber
    &\Pr\Big(\big\|X_{\Pi^c}^\tT X_{\Pi} z\big\|_{2,\infty} \geq \epsilon \|z\|_2\Big)\\
    &\quad\leq e^2 (m-k) \exp\Big(-c_5 \big(\epsilon - \nu_X^g \sqrt{k}\big)^2 (\mu_X^g)^{-2}\Big),
\end{align}
where $c_5 := \frac{(c_1-1)^2}{1024ec_1^2}$ is an absolute constant.
\end{lemma}
\begin{proof}
The proof of this lemma is similar to that of Lemma~\ref{lemma:GroStOC1} and
also relies on Proposition~\ref{prop:azumaineq} in the Appendix. To begin, we
use $\pi_i^c$ to denote the $i$-th element of $\Pi^c$ and note
\begin{align}
    \|X_{\Pi^c}^\tT X_{\Pi} z\|_{2,\infty} \equiv \max_{i \in \nN{m-k}} \big\|\sum_{j=1}^{k} X_{\pi^c_i}^\tT X_{\pi_j} z_j\big\|_2.
\end{align}
We next fix an $i \in \nN{m-k}$ and define $\cA_i^\prime := \{\pi^c_i =
i^\prime\}$ for $i^\prime \in \nN{m-k}$. Then conditioned on $\cA_i^\prime$,
we again have the following simple equality:
\begin{align}
    \nonumber
    &\Pr\Big(\big\|\sum_{j=1}^{k} X_{\pi^c_i}^\tT X_{\pi_j} z_j\big\|_2 \geq \epsilon \|z\|_2 \Big| \cA_i^\prime\Big)\\
\label{eqn:pflemma2_conc}
    &\qquad\qquad\quad= \Pr\Big(\big\|\sum_{j=1}^{k} X_{i^\prime}^\tT X_{\pi_j} z_j\big\|_2 \geq \epsilon \|z\|_2 \Big| \cA_i^\prime\Big).
\end{align}

In order to upper bound \eqref{eqn:pflemma2_conc} using
Proposition~\ref{prop:azumaineq}, we now construct an $\R^r$-valued Doob
martingale $(M_0,M_1,\dots,M_k)$ on $\sum_j X_{i^\prime}^\tT X_{\pi_j} z_j$
as follows:
\begin{align*}
    M_0 &:= \sum_{j=1}^{k} X_{i^\prime}^\tT \E\big[X_{\pi_j}\big|\cA_i^\prime\big]z_j, \quad \text{and}\\
    M_\ell &= \sum_{j=1}^{k} X_{i^\prime}^\tT \E\big[X_{\pi_j}\big|\pi_{1\rightarrow\ell},\cA_i^\prime\big]z_j,\ \ell=1,\dots,k,
\end{align*}
where $\pi_{1\rightarrow\ell}$ denotes the first $\ell$ elements of $\Pi$.
The next step in the proof involves showing $\|M_\ell - M_{\ell-1}\|_2$ is
bounded for all $\ell \in \nN{k}$. To do this, we define
\begin{align}
    M_\ell(u) = \sum_{j=1}^{k} X_{i^\prime}^\tT \E\big[X_{\pi_j}\big|\pi_{1\rightarrow\ell-1},\pi_\ell = u, \cA_i^\prime\big]z_j
\end{align}
for $u \in \nN{k}$ and once again resort to the argument in
Lemma~\ref{lemma:GroStOC1} that $\|M_\ell-M_{\ell-1}\|_2 \leq
\sup_{u,v}\|M_\ell(u) - M_\ell(v)\|_2$. Further, we define an $n \times r$
random matrix
\begin{align}
\nonumber
    \wtX_{\ell,j}^{u,v} := &\E\big[X_{\pi_j}\big|\pi_{1\rightarrow\ell-1},\pi_{\ell}=u,\cA_i^\prime\big]\\
        &\qquad\qquad\quad - \E\big[X_{\pi_j}\big|\pi_{1\rightarrow\ell-1},\pi_{\ell}=v,\cA_i^\prime\big]
\end{align}
and notice that $\wtX_{\ell,j}^{u,v} = 0$ for $j < \ell$,
$\wtX_{\ell,j}^{u,v} = X_u - X_v$ for $j = \ell$, and $\wtX_{\ell,j}^{u,v} =
\frac{1}{m-\ell-1} (X_u - X_v)$ for $j > \ell$. It then follows from this
discussion that
\begin{align}
\nonumber
    \|M_\ell(u) - M_\ell(v)\|_2 &\leq \sum_{j=1}^k \|X_{i^\prime}^\tT \wtX_{\ell,j}^{u,v}\|_2 \|z_j\|_2\\
\label{eqn:pflemma2_MOBD_diffs}
        &\overset{(a)}{\leq} 2\mu_X^g \Big(\|z_\ell\|_2 + \frac{\sum_{j > \ell} \|z_j\|_2}{m-\ell-1}\Big),
\end{align}
where $(a)$ is primarily due to $\|X_{i^\prime}^\tT X_u - X_{i^\prime}^\tT
X_v\|_2 \leq 2\mu_X^g$. We have now established that $\|M_\ell-M_{\ell-1}\|_2
\leq a_\ell$ with
\begin{align}
\label{eqn:pflemma2_diff_bds}
    a_\ell := 2\mu_X^g \Big(\|z_\ell\|_2 + \frac{\sum_{j > \ell} \|z_j\|_2}{m-\ell-1}\Big), \ \ell \in \nN{k}.
\end{align}
The final bound we need in order to utilize Proposition~\ref{prop:azumaineq}
is that on $\|M_0\|_2$. Similar to \eqref{eqn:pflemma1_M0_bd} in
Lemma~\ref{lemma:GroStOC1}, however, it is straightforward to show that
$\|M_0\|_2 \leq \nu_X^g \sqrt{k} \|z\|_2$.

It now follows from our construction of the Doob martingale,
Proposition~\ref{prop:azumaineq} in the Appendix,
\cite[Lemma~B.1]{Donahue.etal.CA1997} and the assumption $k \leq \epsilon^2
(\nu_X^g)^{-2}$ that
\begin{align}
\nonumber
    &\Pr\Big(\big\|\sum_{j=1}^{k} X_{i^\prime}^\tT X_{\pi_j} z_j\big\|_2 \geq \epsilon \|z\|_2 \Big| \cA_i^\prime\Big)\\
\nonumber
    &\qquad\leq \Pr\Big(\big\|M_k - M_0\big\|_2 \geq \big(\epsilon - \nu_X^g \sqrt{k}\big) \|z\|_2 \Big| \cA_i^\prime\Big)\\
\label{eqn:pflemma2_conc_2}
    &\qquad\leq e^2 \exp\Bigg(-\frac{c_0 \big(\epsilon - \nu_X^g \sqrt{k}\big)^2 \|z\|^2_2}{\sum\limits_{\ell=1}^{k} a_\ell^2}\Bigg).
\end{align}
In addition, it can be shown using \eqref{eqn:pflemma2_diff_bds} and the
assumption $k \leq m/c_1$ that $\sum_{\ell=1}^{k} a_\ell^2 \leq
4(1+(c_1-1)^{-1})^2 (\mu_X^g)^2 \|z\|^2_2$. Combining all these facts
together, we obtain the claimed result as follows:
\begin{align}
\nonumber
    &\Pr\Big(\big\|X_{\Pi^c}^\tT X_{\Pi} z\big\|_{2,\infty} \geq \epsilon \|z\|_2\Big)\\
\nonumber
    & \overset{(b)}{\leq} (m-k) \Pr\Big(\big\|\sum_{j=1}^{k} X_{\pi^c_i}^\tT X_{\pi_j} z_j\big\|_2 \geq \epsilon \|z\|_2\Big)\\
\nonumber
    & = \!(m-k)\!\sum_{i^\prime=1}^{m} \Pr\Big(\big\|\sum_{j=1}^{k} X_{i^\prime}^\tT X_{\pi_j} z_j\big\|_2 \geq \epsilon \|z\|_2 \Big| \cA_i^\prime\Big) \Pr(\cA_i^\prime)\\
    & \overset{(c)}{\leq} e^2 (m-k) \exp\Big(-c_3 \big(\epsilon - \nu_X^g \sqrt{k}\big)^2 (\mu_X^g)^{-2}\Big),
\end{align}
where $c_5 := c_0/4(1+(c_1-1)^{-1})^2$, $(b)$ follows from the union bound
and the fact that $\pi_i^c$'s are identically distributed, while $(c)$
follows since $\pi^c_i$ has a uniform distribution over the set $\nN{m}$.
\end{proof}

\begin{figure*}
\centering
\begin{tabular}{ccc}
\includegraphics[width=2.04in]{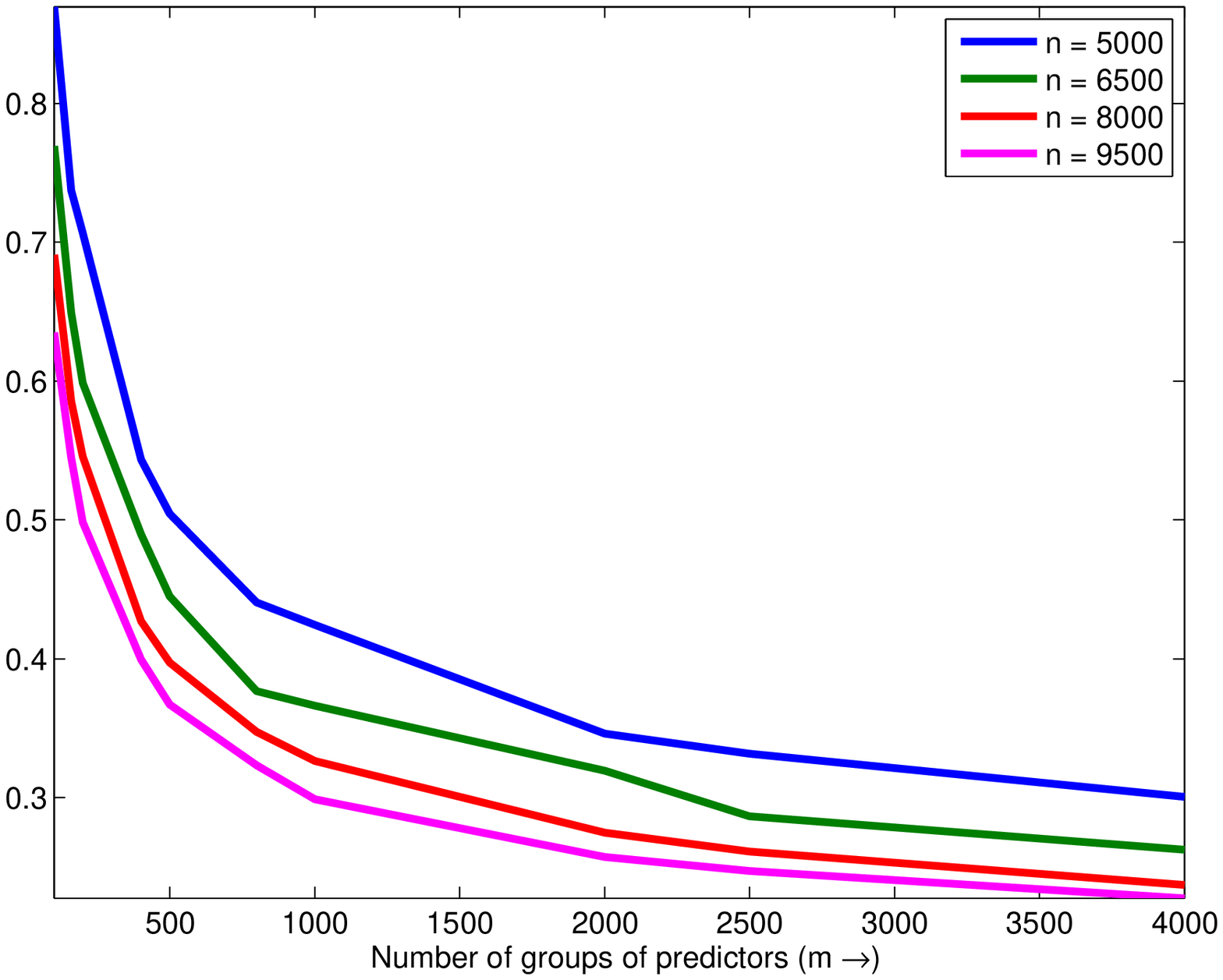}\!\!&
\includegraphics[width=2.04in]{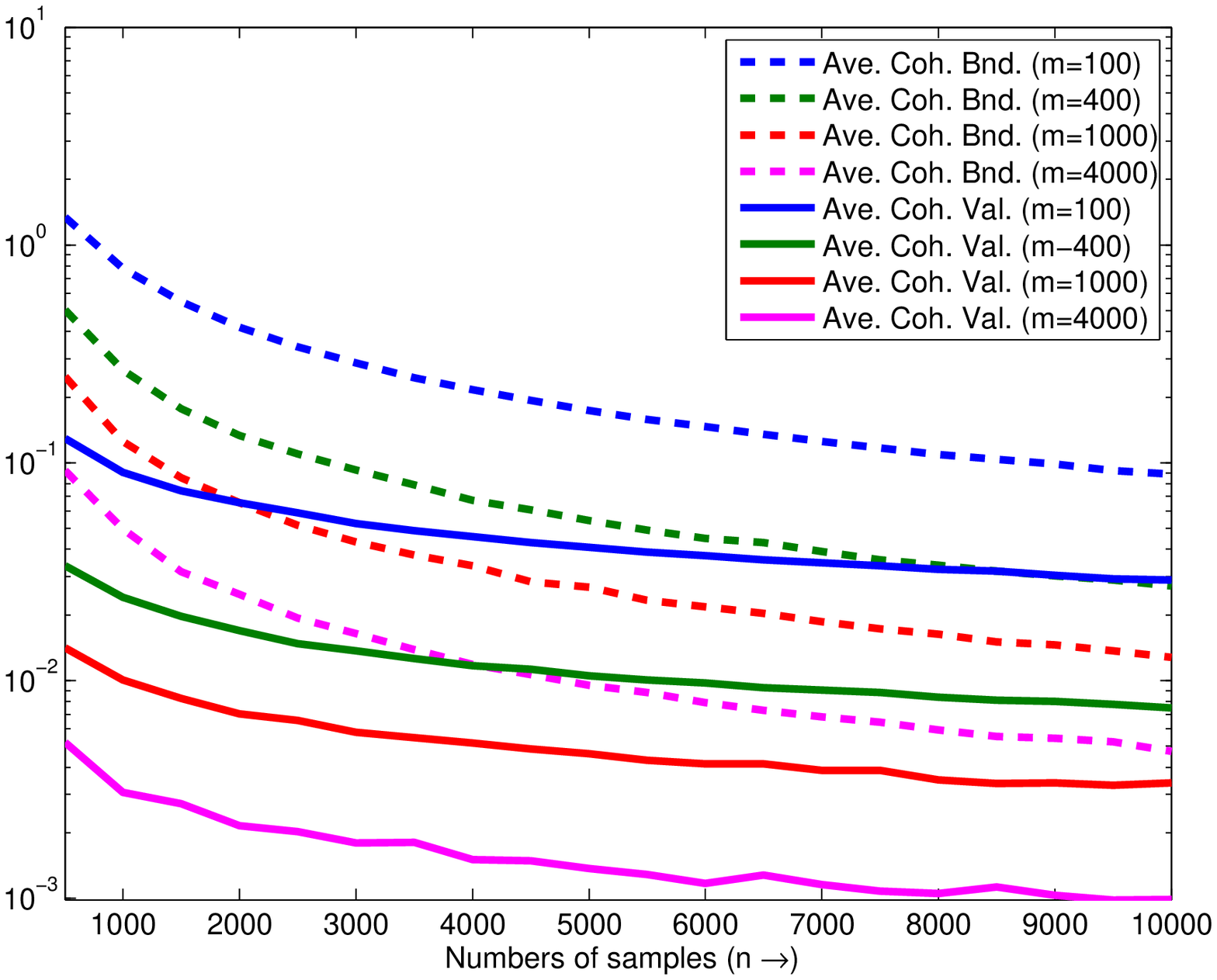}\!\!&
\includegraphics[width=2.04in]{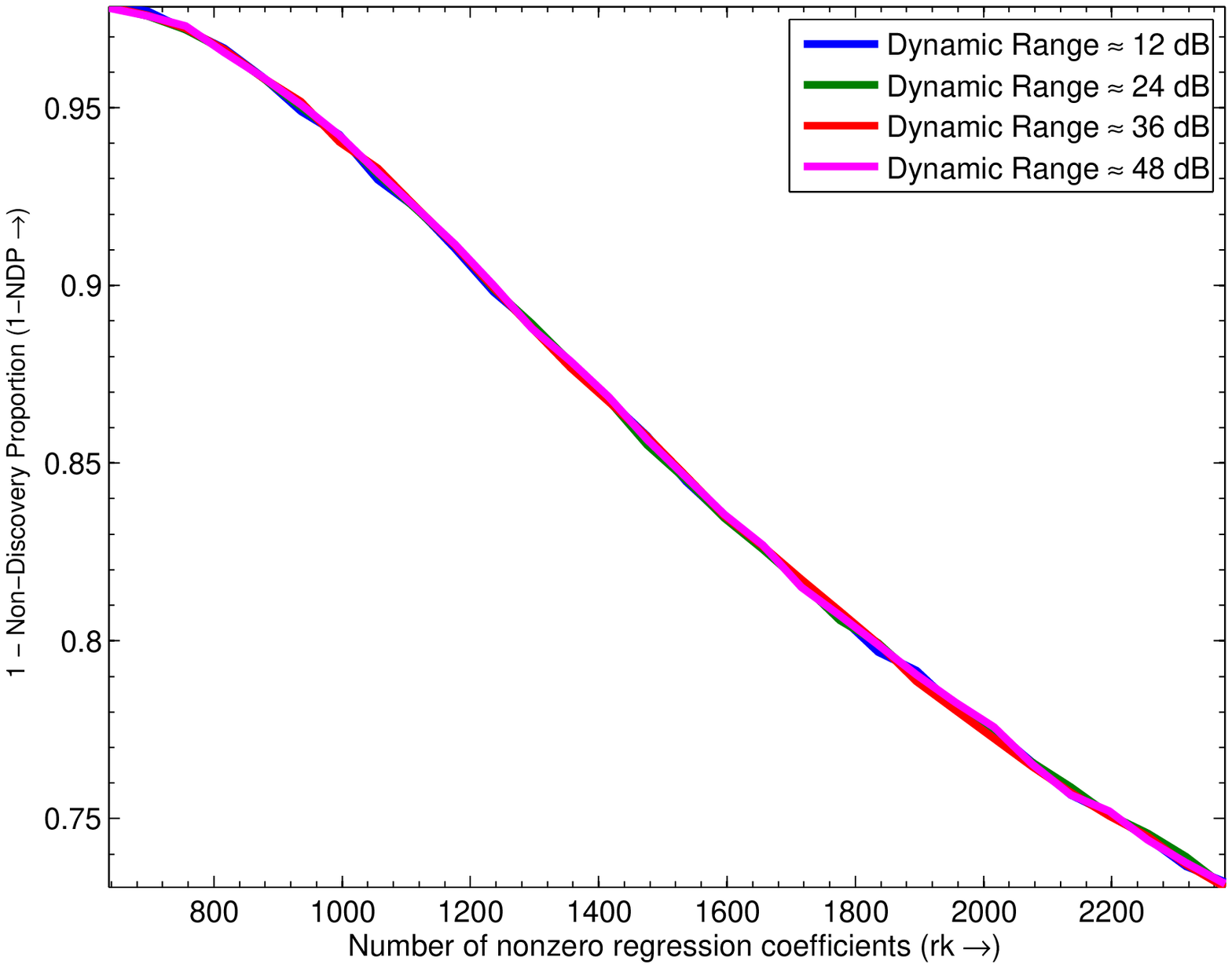}\\
{\footnotesize (a) Plots of $\mu_X^g \sqrt{\log{m}}$} &
{\footnotesize (b) Plots of $\nu_X^g$ (solid) and $\mu_X^g\sqrt{r\log{m}/n}$ (dashed)} &
{\footnotesize (c) Plots of $1 - \NDP$ for GroTh}
\end{tabular}
\caption{Numerical experiments validating main result of the paper. Together,
(a) and (b) illustrate that the set of design matrices satisfying the group
coherence property is not empty. Further, (c) illustrates that the
performance of GroTh is not exactly a function of the dynamic range
$\frac{\max_{i \in \cK} \|\beta_i^0\|_2}{\min_{i \in \cK} \|\beta_i^0\|_2}$.}
\label{fig:numres_1}
\end{figure*}
\begin{figure}
\centering
\includegraphics[width=3.25in]{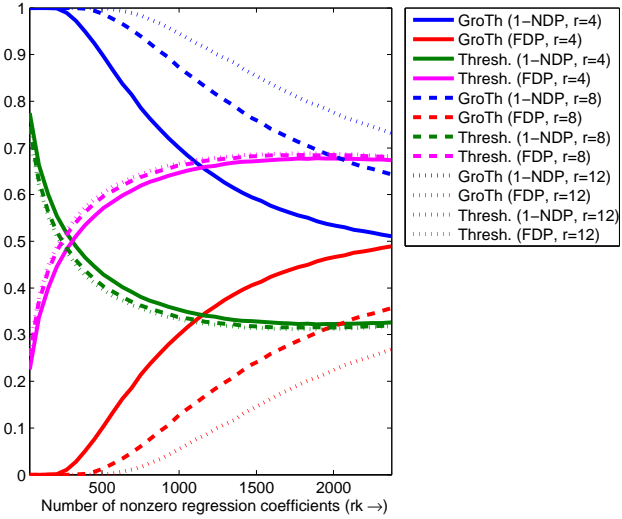}
\caption{Comparison between the performances of GroTh and thresholding of individual marginal correlations that ignores the grouping of predictors.} \label{fig:numres_2}
\end{figure}

\subsection{Proof of Theorem~\ref{thm:groth}}
Define $\wtcK := \Big\{i \in \cK : \|\beta_i^0\|_2 \geq c_3 \mu_X^g
\|\beta^0\|_2 \sqrt{\log{m}}\Big\}$. In order to prove this theorem, we need
to understand the behavior of the marginal correlations corresponding to the
restricted model $\wtcK$ and the marginal correlations corresponding to the
complement set $\cK^c$. To this end, recall the definition of $L$ from the
statement of the theorem and note that
\begin{align}
\nonumber
    \min_{i \in \wtcK} \|f_i\|_2 &= \min_{i \in \wtcK} \|\beta_i^0 + (X_i^\tT X_\cK \beta_\cK^0 - \beta_i^0)\|_2\\
\nonumber
        &\geq \min_{i \in \wtcK} \|\beta_i^0\|_2 - \max_{i \in \wtcK} \|(X_i^\tT X_\cK \beta_\cK^0 - \beta_i^0)\|_2\\
\label{eqn:pfthm_fK_expr}
        &= \|\beta_{(L)}^0\|_2 - \|(X_\cK^\tT X_\cK - I)\beta_\cK^0\|_{2,\infty}.
\end{align}
In addition, we trivially have
\begin{align}
\label{eqn:pfthm_fKc_expr}
    \max_{i \in \cK^c} \|f_i\|_2 = \max_{i \in \cK^c} \|X_i^\tT X_\cK \beta_\cK^0\|_2 = \|X_{\cK^c}^\tT X_\cK \beta_\cK^0\|_{2,\infty}.
\end{align}
It is easy to argue using \eqref{eqn:pfthm_fK_expr} and
\eqref{eqn:pfthm_fKc_expr} that
\begin{align}
\label{eqn:pfthm_suffcond}
    \|\beta_{(L)}^0\|_2 > \|(X_\cK^\tT X_\cK - I)\beta_\cK^0\|_{2,\infty} + \|X_{\cK^c}^\tT X_\cK \beta_\cK^0\|_{2,\infty}
\end{align}
is a sufficient condition for the proof of the theorem. To see this, note
that \eqref{eqn:pfthm_suffcond} implies $\min_{i \in \wtcK} \|f_i\|_2
> \max_{i \in \cK^c} \|f_i\|_2$. This in turn means $\wtcK \subset \whcK$,
since $L \leq k$, resulting in $\FDP(\whcK) \leq 1 - L/k$ and $\NDP(\whcK)
\leq 1 - L/k$.

The next step in the proof is therefore establishing that the sufficient
condition \eqref{eqn:pfthm_suffcond} holds in our case. It is easy to show
using Lemmas~\ref{lemma:GroStOC1}--\ref{lemma:GroStOC2} and the union bound
that
\begin{align}
\label{eqn:pfthm_suffcond2}
    \|(X_\cK^\tT X_\cK - I)\beta_\cK^0\|_{2,\infty} + \|X_{\cK^c}^\tT X_\cK \beta_\cK^0\|_{2,\infty} \geq \epsilon \|\beta^0\|_2
\end{align}
with probability $\delta \leq e^2 m \exp\Big(-c_4 \big(\epsilon - \nu_X^g
\sqrt{k}\big)^2 (\mu_X^g)^{-2}\Big)$ as long as $k \leq \min\{\epsilon^2
(\nu_X^g)^{-2}, c_1^{-1}m\}$ for $c_1 \geq 2$ and $\epsilon \in (0,1)$. We
now fix $\epsilon = c_3 \mu_X^g \sqrt{\log{m}}$ and claim that
\eqref{eqn:pfthm_suffcond2} holds with probability $\delta \leq e^2 m^{-1}$
under the assumptions of the theorem. Notice that validity of this claim
implies the sufficient condition \eqref{eqn:pfthm_suffcond} holds with
probability $1 - \delta \geq 1 - e^2 m^{-1}$ as long as $\|\beta_{(L)}^0\|_2
\geq c_3 \mu_X^g \|\beta^0\|_2 \sqrt{\log{m}}$.

In order to complete the proof, we therefore need only establish the claim
that \eqref{eqn:pfthm_suffcond2} holds for $\epsilon = c_3 \mu_X^g
\sqrt{\log{m}}$ with probability $\delta \leq e^2 m^{-1}$. In this regard,
note: $(i)$ $\epsilon < 1$ because of \eqref{eqn:GroCP-1} with $c_\mu <
c_3^{-1}$ and $(ii)$ $\sqrt{k}\nu_X^g \leq c_2 \epsilon$ because of $c_1 rk
\leq n$ and \eqref{eqn:GroCP-2} with $c_\nu \leq \sqrt{c_1} c_2 c_3$. It then
follows that \eqref{eqn:pfthm_suffcond2} holds for $\epsilon = c_3 \mu_X^g
\sqrt{\log{m}}$ with probability $\delta \leq e^2 m^{1-c_4 (1-c_2)^2 c_3^2}$.
The proof now trivially follows by noting that $c_4 (1-c_2)^2 c_3^2 = 2$ for
the chosen value of $c_3$.\hfill\QEDclosed

\section{Numerical Results}\label{sec:num_res}
In this section, we report the outcomes of some numerical experiments that
validate Theorem~\ref{thm:groth}. The $n \times p$ matrix $X$ in all these
experiments is created as follows. First, we generate $m$ of $n \times r$
matrices $\wtX_i$ whose entries are drawn independently from a standard
normal distribution. Next, we use the Gram--Schmidt process to orthonormalize
$\wtX_i$'s and stack the resulting orthonormal $X_i$'s into an $n \times p$
design matrix $X$.

The first set of experiments reported in Fig.~\ref{fig:numres_1}(a) and
Fig.~\ref{fig:numres_1}(b) confirms that the set of design matrices
satisfying the group coherence property is not empty. Specifically,
Fig.~\ref{fig:numres_1}(a) plots $\mu_X^g \sqrt{\log{m}}$ as a function of
$m$ for $p = 20000$ and four different values of $n$. It can be seen from
this figure that $\mu_X^g \sqrt{\log{m}} = O(1)$, which verifies
\eqref{eqn:GroCP-1}. Further, Fig.~\ref{fig:numres_1}(b) plots both $\nu_X^g$
(solid lines) and $\mu_X^g \sqrt{r\log{m}/n}$ (dashed lines) as a function of
$n$ for $p = 20000$ and four different values of $m$. It can be seen from
this figure that $\nu_X^g = O\big(\mu_X^g \sqrt{r\log{m}/n}\big)$, which
verifies \eqref{eqn:GroCP-2}.

The second set of experiments reported in Fig.~\ref{fig:numres_1}(c) confirms
that the performance of GroTh is not exactly a function of the dynamic range.
In these experiments, corresponding to $p=15000$, $n=3000$ and $r=12$, all
but one group of nonzero regression coefficients $\{\beta_i^0\}_{i\in\cK}$
are normalized to have unit $\ell_2$ norms, while one randomly selected group
of nonzero regression coefficients is normalized to yield specified dynamic
range. Fig.~\ref{fig:numres_1}(c) plots $1 - \NDP$ (averaged over $500$
random realizations of the true model $\cK$) for GroTh under this setup as a
function of $rk$ for four different values of dynamic range. It can be seen
from this figure that the performance of GroTh indeed does not change with
the dynamic range, because of the reasons outlined earlier in
Section~\ref{sec:mainres}.

The final set of experiments reported in Fig.~\ref{fig:numres_2} illustrates
that GroTh performs better than thresholding of the individual marginal
correlations that ignores the grouping of predictors. In these experiments,
corresponding to $p=15000$ and $n=3000$, all groups of nonzero regression
coefficients $\{\beta_i^0\}_{i\in\cK}$ have unit $\ell_2$ norms, but
individual nonzero regression coefficients do not necessarily have same
magnitudes. Fig.~\ref{fig:numres_2} plots $\FDP$ and $1 - \NDP$ (averaged
over $500$ random realizations of the true model $\cK$) for both GroTh and
(individual) thresholding under this setup as a function of $rk$ for three
different values of $r$. It can be seen from this figure that thresholding of
individual marginal correlations performs almost identically for different
$r$. Performance of GroTh, on the other hand, improves with an increase in
$r$.

\section{Conclusions}\label{sec:conc}
In this paper, we have provided a comprehensive understanding of Group
Thresholding (GroTh) for high-dimensional group model selection. In
particular, we have established that the performance of GroTh can be
characterized in terms of a global geometric property of the design matrix
that is explicitly verifiable in polynomial time. Results reported in this
paper have also enhanced our understanding of thresholding-based approaches
in high-dimensional linear models that rely on marginal correlations between
the predictors and the response variable. In the future, we plan on extending
this work by deriving fundamental bounds on worst-case and average group
coherences, providing explicit examples of design matrices that satisfy the
group coherence property, understanding the effects of modeling error, and
relaxing the assumption of orthonormal groups of predictors.

\begin{appendix}[Banach-Space-Valued Azuma's Inequality]
In this appendix, we state a Banach-space-valued concentration inequality
from \cite{Naor.CPaC2012} that is central to this paper.
\begin{proposition}[Banach-Space-Valued Azuma's Inequality]\label{prop:azumaineq} Fix $s > 0$ and assume that a Banach space $(\cB,
\|\cdot\|)$ satisfies
\begin{align*}
    \rho_{\cB}(\tau) := \sup_{\substack{u,v\in\cB\\\|u\|=\|v\|=1}} \left\{\frac{\|u + \tau v\| + \|u - \tau v\|}{2} - 1\right\} \leq s\tau^2
\end{align*}
for all $\tau > 0$. Let $\{M_k\}_{k=0}^{\infty}$ be a $\cB$-valued martingale
satisfying the pointwise bound  $\|M_k - M_{k-1}\| \leq a_k$ for all $k \in
\N$, where $\{a_k\}_{k=1}^{\infty}$ is a sequence of positive numbers. Then
for every $\delta > 0$ and $k \in \N$, we have
\begin{align*}
\Pr\left(\|M_k - M_0\| \geq \delta\right) \leq e^{\max\{s,2\}} \exp\bigg(-\frac{c_0\delta^2}{\sum_{\ell=1}^{k} a_\ell^2}\bigg),
\end{align*}
where $c_0 := \frac{e^{-1}}{256}$ is an absolute constant.
\end{proposition}
\begin{remark}
Theorem 1.5 in \cite{Naor.CPaC2012} does not explicitly specify $c_0$ and
also states the constant in front of $\exp(\cdot)$ to be $e^{s+2}$.
Proposition~\ref{prop:azumaineq} stated in its current form, however, can be
obtained from the proof of Theorem 1.5 in \cite{Naor.CPaC2012}.
\end{remark}
\end{appendix}

\bibliographystyle{IEEEtran}
\bibliography{../../LaTeX/JRCompleteBibDB}
\end{document}